\documentclass[a4paper,12pt]{article}
\usepackage[cp1251]{inputenc}
\usepackage[english]{babel}
\usepackage{amssymb,amsmath,amsthm,cite}
\usepackage[pdftex]{hyperref}
\usepackage{graphics}

\newtheorem{theorem}{Theorem}

\newtheorem{proposition}{Proposition}
\newtheorem{lem}{Lemma}

\theoremstyle{definition}

\newtheorem{definition}{Definition}

\newtheorem{remark}{Remark}
\usepackage{mathtools}

\usepackage[centerlast,small]{caption2}
\begin{document}

\title{Graphs of trigonal curves and rigid isotopies of singular real algebraic curves of bidegree $(4,3)$ on a hyperboloid}
		
	\author{V.~I. Zvonilov\thanks{The work was done on a scientific project of the State assignment FSWR-2023-0034.}}
%
\maketitle	
	
	\begin{abstract} 
		A rigid isotopy of  real algebraic curves of a certain class
		is  a path in the space of curves of this class. The paper's study completes the rigid isotopic classification of nonsingular real algebraic curves of
		bidegree (4,3) on a hyperboloid, begun by the author in earlier works.   There are given the missing proofs of the uniqueness of the connected components for 16 classes of real algebraic curves of bidegree (4,3)
		having a single node or a cusp.  
		The main technical tools are graphs of real trigonal curves on Hirzebruch surfaces.
	\end{abstract}
	
	

\newcommand{\R}{\mathbb{R}}
\newcommand{\Cb}{\mathbb{C}}
\newcommand{\Z}{\mathbb{Z}}
\newcommand{\ti}{\tilde}

\newcommand{\De}{\Delta}
\newcommand{\<}{\langle}
\newcommand{\ra}{\rangle}
\newcommand{\Sk}{\operatorname{Sk}}
\newcommand{\Skdir}{\Sk_{\fam0 dir}}
\newcommand{\Skud}{\Sk_{\fam0 ud}}

\def\inserthyphen{\ifcat\next a-\fi\ignorespaces}
\def\pblack-{$\bullet$\futurelet\next\inserthyphen}
\def\pwhite-{$\circ$\futurelet\next\inserthyphen}
\def\pcross-{$\vcenter{\hbox{$\scriptstyle\times$}}$\futurelet\next\inserthyphen}
\def\black{\protect\pblack}
\def\white{\protect\pwhite}
\def\cross{\protect\pcross}
 \section{Introduction}
In the author's paper \cite[Theorem 2]{Z99} it was asserted that {\it a non-singular curve of bidegree} (4,3)
{\it on a hyperboloid is determined up to a rigid
	isotopy by its complex scheme}. To prove this, it was used the approach proposed in \cite{DIK00}
for obtaining a rigid isotopic classification of plane real quintics. The proof is based on Theorem 1 of \cite{Z99}, which specifies all the connected components of the space
of bidegree (4,3) curves that have a unique non-degenerate double point or a
cusp. In the present paper, this Theorem 1 is proved in a different way -- using graphs of real trigonal curves on the Hirzebruch surface $\Sigma_3$ obtained from singular curves of bidegree (4,3). The proof of this theorem was started by the author in \cite{Z99}; in \cite{Z03} a gap in this proof was pointed out, which is eliminated in the present paper.

The structure of the paper: Section \ref{hyp} contains the necessary information on the topology of real algebraic curves, available in \cite{R}, \cite{V}, and in particular, of curves on a hyperboloid (see \cite{Z91}). Section \ref{trig}, following the papers \cite{DIK08} and \cite{Z21}, recalls the concepts associated with real trigonal curves. Section \ref{Nagata} establishes the connection between curves on a hyperboloid and trigonal curves. In section \ref{deg9} it is filled the gap indicated above, which consists in the absence of a proof of the fact that  each of the classes of singular curves $\omega^{\pm}_{inn}$, $\alpha^{\pm}_{lp}<l>, 0\leq l\leq5,$ (in the notation of works \cite{Z99}, \cite{Z03}) on the hyperboloid is connected. The same arguments give a proof of the uniqueness of the remaining classes of singular curves.
\section{Definitions and notation}\label{hyp}
\emph{A real algebraic variety}
is a complex algebraic variety~$V$ with
an antiholomorphic involution $c = c_V: V \rightarrow V$.
The set of fixed points ${\R} V = \mathrm{Fix}\, c$ is called the \emph{real part} of the variety~$V$. A regular morphism $f:V \rightarrow W$ of two real manifolds
is called \emph{real} or \emph{equivariant}
if $f\circ c_V=c_W\circ f$

A hyperboloid $X$ is understood as a real non-singular quadric with complex part $\Cb X=\Cb P^1\times \Cb P^1$, an antiholomorphic
involution ${c_X}$ that preserves the product, and real part $\R X=\R P^1\times \R P^1$.

Fix a pair $P_1, P_2$ of generators of the hyperboloid $X$. The fundamental classes
$[\Cb P_1]$, $[\Cb P_2]$ form a basis of the group $H_2(\Cb X)\cong \Z \oplus \Z$.
Let $A$ be an algebraic curve on $X$. Then $[\Cb A]=
m_1 [\Cb P_1]+m_2 [\Cb P_2]$ for some non-negative integers $m_1$, $m_2$.
The pair $(m_1,m_2)$ is called the {\it bidegree} of $A$. If $[x_0:x_1]$,
$[y_0:y_1]$ are homogeneous coordinates on the lines $P_1$, $P_2$, then the curve $A$
is defined by a bihomogeneous polynomial
$$F(x_0,x_1;y_0,y_1)=\sum_{i,j=1}^{m_1,m_2} a_{i,j}x_1^i x_0^{m_1-i}y_1^j y_0
^{m_2-j}, $$
having the degrees of homogeneity $m_1$ in $x_0, x_1$ and $m_2$ in $y_0, y_1$. The curve $A$ is real
if and only if all $a_{i,j}$ are real.

To specify the topology of a real curve on a hyperboloid, we use a modification of the standard encoding of real schemes of plane projective curves (see,
e.g., \cite{V}).
Let $A\subset X$ be a nonsingular
real algebraic curve. The real part $\R A$ may have components of two types: contractible in $\R X$ and non-contractible; contractible components
are called {\it ovals}. The number of ovals is denoted by $l$, the number of
non-contractible components by $h$. Each oval bounds a topological
disk in $\R X$ called the {\it interior} of the oval. Fundamental classes $[\R P_1]$, $[\R P_2]$,
endowed with some (fixed) orientations, form a basis of the group
$H_1(\R X)\cong
\Z \oplus \Z$. All non-contractible components $N_1,...,N_h$ realize the same
non-zero class $(c_1,c_2)$ in $H_1(\R X)$, where $c_1$, $c_2$ are coprime. The real scheme of the curve $\R A\subset \R X$ is encoded as follows:
$\langle (c_1,c_2), {\rm scheme}_1,(c_1,c_2), {\rm scheme}_2,...,(c_1,c_2), {\rm scheme}_h
\rangle,$
where ${\rm scheme}_1,..., {\rm scheme}_h$ are the arrangement schemes of the ovals lying in the connected
components of the surface
$\R X\setminus (N_1\cup ... \cup N_h)$ (cf. \cite{V}, \cite{Z91}).

According to F. Klein (see \cite{Kl} or \cite{R}) a real curve $A$
belongs to {\it type} I
or {\it type} II depending on whether $\R A$ divides the complexification
$\Cb A$ or not. If $A$ belongs to type I, the natural orientations of the components
$U$ and $V$ of the space $\Cb A\setminus \R A$
give two opposite orientations of the curve $\R A= \partial U= \partial V$;
these are called {\it complex orientations}. A real scheme endowed with
a type and, in the case of type I, complex orientations, is called the {\it
	complex scheme}.
If one wishes to specify the type of a curve with a real scheme
$\<B\ra$, one uses the notation $\<B\ra_{\rm I}$ and
$\<B\ra_{\rm II}$.

All real curves of bidegree $(m,n)$ form a space
$C_{m,n}\cong \R P ^N$
with $N=mn+m+n$. The set $\De \subset C_{m,n}$ of singular curves has the dimension
$N-1$. By $S\subset \De$ we denote the subset of curves that have a singular
point other than a non-degenerate double point or cusp, or
have several singular points. The set $\De \setminus S$ is a topological manifold (though not a smooth submanifold
of
$C_{m,n}$). {\it A rigid isotopic class} of a curve
$A\in C_{m,n}\setminus \De$ (or $A\in \De \setminus S$) is a component
of the space $C_{m,n}\setminus \De$ (resp. $\De \setminus S$) containing $A$.
The components of the space $C_{m,n}\setminus \De$ (resp. $\De \setminus S$) are called
{\it chambers (walls)}.
\section{Hirzebruch surfaces and Nagata transformations}\label{Nagata}
A Hirzebruch surface $\Sigma_k, k\geq0$ is the space of the line bundle $q:\Sigma_k\rightarrow P^1$, i.e. a real rational ruled surface with exceptional real section $E_k$, $E_k^2=-k$. The fibers of the bundle~$q$ are called
\emph{vertical}, for example, we speak of vertical tangents, vertical
flexes, etc. The surface $\Sigma_0=P^1\times P^1$ is a hyperboloid in which any curve of bidegree $(0,1)$ can be taken as an exceptional section. When studying the walls of the space $C_{m,n}$, the exceptional section is considered to be a curve of bidegree $(0,1)$ passing through a singular point of the curve from the wall.

\label{Nag}
\begin{definition}
	\emph{A positive (negative) Nagata transformation} (see \cite[\S\,2 (3)]{N}) is a fiberwise birational transformation of $\Sigma_k\rightarrow\Sigma_{k+ 1}$ (respectively, $\Sigma_k\rightarrow\Sigma_{k-1}$),
	consisting of blowing up a point $p\in E_k$ (respectively, $p\notin E_k$) and then contracting it to a point of the proper transform
	of the fiber $q^{-1}(p)$.
\end{definition}
A positive (negative) Nagata transformation takes $E_k$ to $E_{k+1}$ (respectively, to $E_{k-1}$). 
 
 \section{Real trigonal curves}\label{trig} 

\emph{A trigonal curve} is a reduced curve $C\subset\Sigma_k$ that does not contain the exceptional section or a fiber as a component and such that the restriction $q|_C$ is a mapping of degree three. A trigonal curve $C$ is called \emph{proper} if $C\cap E_k=\varnothing$.

A fiber $F$ of the surface $\Sigma_k$ with a trigonal curve $C\in\Sigma_k$  is called \emph{singular} if $F$ intersects $C\cup E_k$ geometrically in fewer than four points.

According to \cite[p. 3.1.1]{Degt} for a proper trigonal curve $C$ on the surface
$\Sigma_k, k\geq1,$ there exists an affine chart $(x,y)$ on this surface in which the exceptional section $E_k$ and the curve $C$ are defined by the equation $y=\infty$  and by the  \emph{Weierstrass equation}

$$\,y^3+b(x)y+w(x)=0,$$
where $b, w$ are real polynomials, $\deg b \leq 2k, \deg w \leq 3k$. We can extend the chart to $\Sigma_k\setminus E_k$, assuming that $b, w$ are homogeneous polynomials of degrees $2k, 3k$.


A non-singular proper real trigonal curve $ C $ is called \emph{almost generic}
if all critical points of the restriction $ q|_C $, i.e. all roots of the discriminant
$ \Delta (x) = 4b^3 + 27w^2 $, are simple, and \emph{maximally inflected} if all roots of the discriminant are real (and not necessarily simple).

A real proper trigonal curve $ C $ (possibly singular) is called
\emph{hyperbolic}
if the restriction $ { \R } C \rightarrow { \R } P^1 $ of the map~$ q $ is a three-sheeted covering.

For a real trigonal curve $ C $, the rational function $ j = j_C = \frac {4b^3} { \Delta } = 1 - \frac { 27w^2} { \Delta } $ is called the \emph{ $ j $-invariant} of this curve. A curve with a constant $j$-invariant is called \emph{isotrivial}.
An almost generic curve is called \emph{generic} if for each real critical value $t$ of its $j$-invariant the multiplicity of each root of the equation ~$j(x)=t$ is equal to ~$3$ for $t=0$, is equal to ~$2$ for $t\neq 0$ and all these roots are real for $ t\notin \{0,1\}$. Any almost generic trigonal curve can be transformed into a generic one by a small change in the coefficients of the curve equation.

The real part of a nonsingular non-hyperbolic curve~$C$ has a unique
\emph{long component} $l$, characterized by the fact that the restriction $l \rightarrow {\R} P^1$ of the mapping~$q$
has degree~$\pm 1$. For all other components of~${\R} C$,
called \emph{ovals}, this degree is~$0$. Let $Z \subset {\R} P^1$ be a set of points
with more than one preimage in ${\R} C$. Each oval is mapped by $q$
to a whole component of~$Z$, which is also called an
\emph{oval}. The remaining components of~$Z$,
as well as
their preimages
in~$l$, are called \emph{zigzags}.

By a \emph{deformation} of a trigonal curve
$C\subset\Sigma_k$ we mean a deformation
of the pair $(q:\Sigma_k\rightarrow P^1, C)$
in the sense of Kodaira-Spencer.
A deformation of an almost generic trigonal curve
is called \emph{fiberwise} if the curve remains
almost generic throughout the deformation.

\emph{Deformation equivalence} of real trigonal curves is defined by the equivalence relation generated by equivariant fiberwise deformations and real isomorphisms.

\subsection{Special fibers and Nagata transformations}
To obtain a proper trigonal curve on $\Sigma_3$ from a singular curve $C\in\De \setminus S$ of bidegree $(4,3)$ on a hyperboloid using the Nagata transformation, consider a chart $(x,y)$ on it, where the exceptional section $E_0$ and the singular fiber $F$ are given by the equations $y=\infty$ and $x=0$. Then the positive Nagata transformation $N$ centered at the point $p=F\cap C\cap E_0$ is given by the equality $(x,z)=(x,xy)$. The curve $C$ and its image $N(C)$ are given by local equations with only the necessary initial terms indicated. We need the Nagata transformations of the following
singular fibers (the fiber of the curve $N(C)$ is designated according to \cite[3.1.2]{Degt}):
\begin{enumerate}\label{list}
	\item $p$ is a non-degenerate double point at which the curve $C$ is tangent to neither $F$ nor $E_0$. $C: x^2y^3+y+1=0$, $N(C): z^3+z+x=0$, fiber $\tilde{A}_0$;
	\item $p$ is a non-degenerate double point at which the curve $C$ is tangent to the fiber $F$ and not tangent to $E_0$. $C: x^2y^3+xy^2+1=0$, $N(C): z^3+z^2+x=0$, fiber $\tilde{A}_0^*$;
	\item $p$ is a non-degenerate double point where $C$ is tangent to $E_0$ and not tangent to $F$. $C: x^3y^3+xy^2+y+1=0$, $N^2(C): z^3+z^2+xz+x^3=0$, fiber $\tilde{A}_1$;
	\item $p$ is a non-degenerate double point where $C$ is tangent to both $E_0$ and $F$. $C: x^3y^3+xy^2+1=0$, $N^2(C): z^3+z^2+xz+x^3=0$, fiber $\tilde{A}_2$;
	\item $p$ is a cusp where the tangent line coincides with neither $E_0$ nor $F$.
	$C: x^2y^3+2xy^2+y+x^2y^2+1=0$, $N(C): z^3+2z^2+z+xz^2+x=0$, fiber $\tilde{A}_0^*$;
	\item $p$ is a cusp with vertical tangent. $C: x^2y^3+1=0$, $N(C): z^3+x=0$, fiber $\tilde{A}_0^{**}$;
	\item $p$ is a cusp with horizontal tangent. $C: x^3y^3+y+1=0$, $N^2(C): z^3+xz+x^3=0$, fiber $\tilde{A}_1^{*}$;
	\item $p$ is a non-singular point of $C$ with a non-vertical and non-horizontal tangent, and $F$ intersects $C$ in three distinct points. $C: xy^3+y^2+1=0$, $N(C): z^3+z^2+x^2=0$, fiber $\tilde{A}_1$;
	\item $p$ is a non-singular point of $C$ with a non-vertical and non-horizontal tangent, and $F$ is tangent to $C$ (at a point different from $p$). $C: xy^3+y^2+x=0$, $N(C): z^3+z^2+x^3=0$, fiber $\tilde{A}_2$;
	\item $p$ is a non-singular point of $C$ with a vertical tangent, and $F$ intersects $C$ at a point different from $p$. $C: xy^3+y+x=0$, $N(C): z^3+xz+x^3=0$, fiber $\tilde{A}_1^*$;
	\item $p$ is an inflection point of $C$ with vertical tangent. $C: xy^3+1=0$, $N(C): z^3+x^2=0$, fiber $\tilde{A}_2^*$.
	
\end{enumerate}

\subsection{The graph of a trigonal curve}\label{lift}
 
Further we shall need graphs on the disk (\emph{dessins}, as special cases of trichotomic graphs, in the terminology of \cite[Section 5]{DIK08}, 
\cite{JP}), 
isomorphic to the graphs of real trigonal curves (see the definition in the next paragraph). According to \cite{DIK08}, any such graph is the graph of some trigonal curve. Unless otherwise stated, everywhere below the \emph{graph} is the graph of a real trigonal curve.

Let $ D $ be the disk obtained by factorizing the complex projective line $\Cb P^1 $ by complex conjugation, and $\mathrm{pr}: \Cb P^1 \rightarrow D$ be the projection. Points, segments, etc., lying
on the boundary circle~$\partial D$ are called \emph{real}.
For $ j $-invariant $j_C:\Cb P^1\rightarrow \Cb P^1= \Cb\cup\{\infty\}$ of a non-isotrivial real trigonal curve $ C\subset\Sigma_k $ we endow the line $ \R P^1$ lying in the image of this function with the orientation determined by the order in $\R$ and color it as follows: let $0$, $1$,
and~$\infty$ be, respectively, the \black--, \white--, and
\cross-- vertex; $(\infty,0)$, $(0,1)$, and $(1,\infty)$ be, respectively, the \rm{solid}, \rm{bold}, and \rm{dotted} edge.
Lift this orientation and coloring to the graph $\Gamma_C=\mathrm{pr}(j_C^{-1}({\R} P^1))$, obtaining the \emph{graph of a real trigonal curve} $C$. Its \black--, \white--, and
\cross--vertices, which are branch points (critical points of the $j$-invariant) with critical values $ 0,1$ and $\infty$, are called \emph{essential}, the remaining vertices, which are
branch points with real critical values different fro$0,1,\infty$, are called \emph{monochrome}. 
	Monochrome vertices are
	classified as solid, bold, and dotted, according to the edges that adjoin them.
	
	A \emph{monochrome cycle} in~$\Gamma_C$ is a cycle all of whose vertices
	are monochrome; hence, all of its edges and vertices are of the same color.
	The definition of the graph~$\Gamma_C$ implies
	that it has no directed monochrome cycles.
	
	The \emph{degree} of $\Gamma_C$ is~$3k$. A graph of degree~$3$ is called \emph{cubic}.
	
	In the figures, the \emph{real part} $ \partial D \cap\Gamma $ of the graph $ \Gamma $ and its subsets
	are denoted by wide
	gray lines.
	
	For a graph
	$\Gamma\subset D$, 
	the closures of the connected components of the set~$D\setminus\Gamma$
	are called \emph{regions} of~$\Gamma$. A region with three essential vertices on its
	boundary is called \emph{triangular}.
	
	The graph of a curve is called
	\emph{unramified} if all its \cross--vertices
	are real.
	In other words, unramified graphs are those corresponding to
	maximally inflected curves.
\subsection{Pillars}\label{pillar}
A graph $\Gamma$
is called
\emph{hyperbolic} if all its real edges are dotted. It corresponds to a hyperbolic curve.

The union of closures of some identically colored real edges of $\Gamma$ is called a \emph{segment} if it is
homeomorphic to a segment.
A dotted (bold)
segment
is called \emph{maximal} if its endpoints are
two
\cross--vertices
(respectively, two \black--vertices).
A \emph{dotted}/\emph{bold} \emph{pillar} is a
maximal dotted/bold segment.

Ovals and zigzags of a non-hyperbolic graph $\Gamma$ are depicted in $\partial D$
by dotted pillars that contain an even and an odd number of \white--vertices, respectively.

A bold pillar
with an even/odd number of
\white--vertices
is called a
\emph{wave/jump}.

\subsection{Elementary moves of graphs}\label{graphmodif}

Two graphs are said to be
\emph{equivalent} if, up to a homeomorphism $f$
of the disk $D$, they can be connected by
a finite sequence of
isotopies and the following \emph{elementary moves}:
\begin{itemize}
	\item[--]
	\emph{monochrome modification}, see
	\ref{fig.moves}(a);
	\item[--]
	\emph{creation} (\emph{destroying}) \emph{a bridge}, see
	\ref{fig.moves}(b),
	where a \emph{bridge} is a pair of
	monochrome vertices connected by a real monochrome edge;
	\item[--]
	\emph{\white-in} and its inverse \emph{\white-out}, see
	\ref{fig.moves}(c) and~(d);
	\item[--]
	\emph{\black-in} and its inverse \emph{\black-out}, see
	\ref{fig.moves}(e) and~(f).
\end{itemize}
(In the first two cases, a move
is considered possible only if it results
in a graph without directed monochrome cycles.)
An equivalence of two graphs is called \emph{bounded} if the homeomorphism $f={\rm id}_D$ and the above isotopies preserve the pillars as sets.

\begin{figure}[tb]
	\begin{center}
		\includegraphics{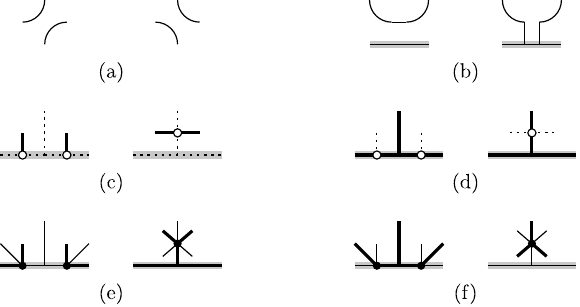}\\
	\end{center}
	\caption{Elementary moves of graphs}\label{fig.moves}
\end{figure}

The following assertion is proved in the same way as the similar assertion~\cite[5.3]{DIK08}.

\begin{theorem} \label{equiv.curves}
	Two generic real trigonal curves are deformation	
	equivalent
	\emph{(}in
	the
	class
	of
	almost generic
	real trigonal curves\emph{)}
	if and only if their graphs are equivalent.
\end{theorem} 

\subsection{Rigid isotopies and weak equivalence}\label{s.rigid}
The rigid isotopy differs from the deformation equivalence by an additional pair of mutually inverse operations:
straightening/creating a zigzag, the former consisting in
merging two vertical tangents bounding the zigzag into a single vertical flex, followed by pulling them apart
to the imaginary domain. 
At the graph level,
these operations are shown in \ref{fig.zigzag}.

\begin{definition}
	Two graphs are called \emph{weakly equivalent} if they
	are related by a sequence of isotopies,
	elementary moves (see~\ref{graphmodif})
	and the operations of
	\emph{straightening/\penalty0creating a zigzag}
	consisting in replacing one of the fragments shown in
	\ref{fig.zigzag} with another.
\end{definition}

\begin{figure}[tb]
	\begin{center}
		\includegraphics{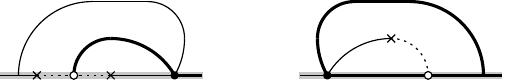}\\
	\end{center}
	\caption{Straightening/Creating a Zigzag}\label{fig.zigzag}
\end{figure}
The following proposition is easy to deduce from~\cite{DIK08}.

	Two generic real trigonal curves are rigidly isotopic.
	if and only if their graphs are weakly equivalent.

As mentioned in \cite{DIZ}, the following theorem
can be deduced, for example, from
Propositions~5.5.3 and~5.6.4 in~\cite{DIK08},
see also~\cite{Z06}.
\begin{theorem} \label{max.inflected}
	Any non-hyperbolic nonsingular
	real trigonal curve
	on the Hirzebruch surface
	is rigidly isotopic to a maximally inflected one.
\end{theorem}
In the framework of weak equivalence of graphs, we need the following two transformations that reduce the number of real \white--vertices of a graph:
\begin{enumerate}
	\item removal of adjacent jumps and zigzags, consisting of straightening the zigzag followed by \white-- and \black-in, and the inverse transformation (see \ref{JZ});
	\item removal of a pair of adjacent zigzags, consisting of \black-out, and straightening the two zigzags followed by \white-- and \black-in, and the inverse transformation (see \ref{ZZ}).
\end{enumerate}
\begin{figure}
	\begin{center}
		\includegraphics{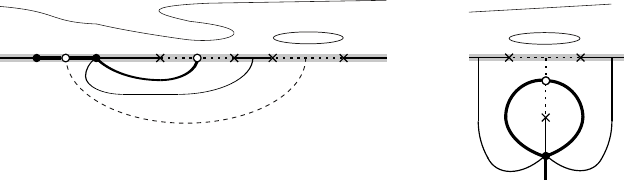}\\
	\end{center}
	\caption{Removing/creating adjacent jump and zigzag}\label{JZ}
\end{figure}
\begin{figure}
	\begin{center}
		\includegraphics{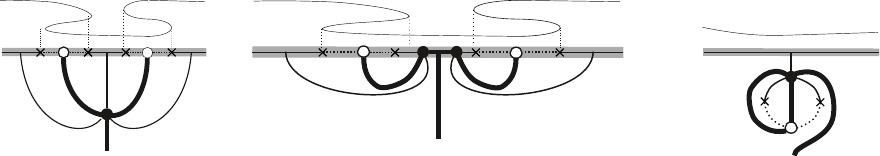}\\
	\end{center}
	\caption{Removing/creating a pair of adjacent zigzags}\label{ZZ}
\end{figure}

 \subsection{Singular curves}
 A non-isotrivial trigonal curve is called \emph{nodal-cuspidal} if all roots of its discriminant $\Delta$ have multiplicity at most three.
 
 Theorem \ref{equiv.curves} extends to real nodal-cuspidal curves and their graphs
 (see \cite[Theorem 4.26]{Degt}).
 
 It is easy to check that  \cite[Theorem 3]{Z21} extends to real nodal-cuspidal curves that have no imaginary singularities:
 \begin{theorem} \label{s.max.inflected}
 	Any real nodal-cuspidal non-hyperbolic curve without imaginary singularities is rigidly isotopic to a maximally inflected curve.
 \end{theorem}
 For the graph of a real nodal-cuspidal curve, the vertices   that correspond to singular points of the curve are called \emph{singular}. We include in the number of dotted pillars those real singular \cross--vertices that are adjacent to solid edges, i.e. corresponding to isolated singular points of the curve ("degenerate ovals"). We call such singular vertices \textit{isolated}.
 
 \begin{figure}
 	\begin{center}
 		\scalebox{1.8}{\includegraphics{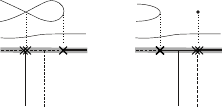}}\\
 	\end{center}
 	\caption{Passage of a nodal point through a cusp}\label{isol}
 \end{figure}
 \begin{proposition}\label{s-oval}
 	On a real segment without \white--vertices, an isolated singular point can be swapped with an oval.
 \end{proposition}
 \begin{proof}
 	On such a segment, all \black--vertices occur in pairs, which can be removed from the segment using \black-in. After that, the \ref{isol} transformation swaps the isolated singular point and the oval.
 \end{proof}
 \begin{figure}
 	\begin{center}
 		\includegraphics{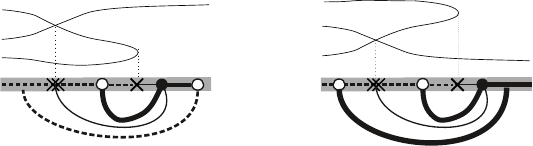}\\
 	\end{center}
 	\caption{ Passage of a nodal point through a point $\tilde{A}_1^*$}\label{A1}
 \end{figure}
 \subsection{Cuts}\label{cuts}
 For $j=1,2$, let $D_j$ be a disk, $\Gamma_j\subset D_j$ be the graph of a nodal-cuspidal curve, $I_j\subset\partial D_j$ be a segment and 
 $\varphi:I_1\to I_2$ be an isomorphism, i.e., a diffeomorphism of segments
 establishing an isomorphism of the graphs
 $\Gamma_1\cap I_1\to\Gamma_2\cap I_2$.
 Consider
 the quotient set $D_{\varphi}=(D_1\sqcup D_2)/\{x\sim\varphi(x)\}$ and the image
 $\Gamma'_{\varphi}\subset D_{\varphi}$ of the graph~$\Gamma_1\cup\Gamma_2$. Denote by~$\Gamma_{\varphi}$
 the graph obtained from~$\Gamma'_{\varphi}$ by deleting the image of the segment
 ~$I_1$ if $\varphi$ changes orientation, and if it does not, then either by transforming
 the images of the endpoints of~$I_1$ into monochrome vertices, or by preserving these endpoints as essential vertices.
 
 In what follows, we always
 assume that $I_j$ is part of an edge of the graph ~$\Gamma_j$, or
 $I_j$ contains one
 \white--vertex, or it ends at singular vertices (and then contains one monochrome vertex).
 
 Up to isotopy, in the second and third cases $\varphi$
 is unique; in the first case, this also requires specifying whether $\varphi$ preserves orientation or reverses it.
 If $\Gamma_{\varphi}$ is a trigonal curve graph, it is called the result of
 \emph{gluing} the graphs $\Gamma_1$, $\Gamma_2$ along~$\varphi$.
 The image of ~$I_1$ is called a \emph{cut}
 in~$\Gamma_{\varphi}$.
 A cut
 is called \emph{genuine} (\emph{artificial}) if $\varphi$ preserves
 (respectively, reverses) orientation; it is called \rm{solid},
 \rm{dotted} or \rm{bold} depending on the structure of the segment
 $\Gamma\cap I_1$.
 (The names \rm{dotted} and \rm{bold} are still applied
 to cuts containing a
 \white--vertex.)
 A \emph{junction} is a genuine cut obtained by gluing two graphs
 along isomorphic
 parts of their zigzags.
 \section{A constructive description	of maximally inflected trigonal curves}\label{S.rational}
 In this section, we give a constructive description of the
 real parts of nonsingular maximally inflected
 trigonal curves.
 
 \subsection{Blocks}\label{blocks}
 \begin{definition}\label{def.block}
 	\emph{A cubic block of type~$I$} is an unramified graph of degree~$3$ and of type~$I$ (see \ref{cubics} I). \emph{A cubic block of type~$II$} is an unramified graph of degree~$3$ and of type~$II$ with an interior \black--vertex (see \ref{cubics} II). Several cubic blocks, artificially glued along segments of solid edges, form a (\emph{general}) \emph{block}. 
 	The block type is the type of the curve corresponding to the block.
 \end{definition}
 \begin{figure}[tb]
 	\begin{center}
 		\includegraphics{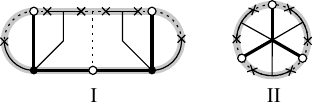}\\
 	\end{center}
 	\caption{Cubic blocks}\label{cubics}
 \end{figure}
 
 According to \cite[5.6.7]{DIK08} and \cite[5.1]{Z21}, both cubic blocks are unique up to isomorphism.
 \begin{remark}\label{replace}
 	From the description of the blocks {\rm \cite[Proposition 8]{Z21}} it immediately follows that the genuine gluing of two blocks along bold segments can be replaced either by a junction of two other blocks composed of parts of the previous blocks, or by artificial gluing along segments of solid edges, i.e. by a new block.
 \end{remark} 
 \subsection{The real parts of maximally	inflected curves}\label{CR}
 A complete description of the real part of a maximally
 inflected nonsingular
 trigonal curve,
 \emph{i.e.} a description of the topology of $\R C\cup s_0$, where $s_0\subset\Sigma_k$ is the zero section, is given in \cite[5.3]{Z21} and, taking into account the remark \ref{replace}, looks as follows:
 the graph of a maximally
 inflected curve is obtained from a disjoint union of blocks using junctions that transform the disks of the blocks into a single
 disk.
 In this case, if all blocks are of type I and all gluings are junctions, the result is a graph of type I, otherwise is a graph of type II.
 \section{Skeletons}\label{S.skeletons}
Here the notion of a skeleton, introduced in \cite{Z21} for maximally inflected trigonal curves, is extended to the case of the graph of a curve obtained from a maximally inflected one by the transformations \ref{JZ} and \ref{ZZ} (see below  the definition of \ref{def.skeleton}) 
\subsection{Abstract skeletons}\label{a.skeletons}
Consider an embedded (finite) graph $\Sk \subset D$ in a
disk $D$. We do not exclude the possibility that
some vertices~$\Sk$ belong to the boundary~$D$;
such vertices are called
\emph{real}, the rest are called \emph{imaginary} or \emph{inner}. The set of edges adjacent to each real 
vertex~$v$ of the graph~$\Sk$ receives a pair of opposite linear
orders from the pair of orientations of the circle~$\partial D$ .

The \emph{Immediate neighbors} of an edge~$e$ at a vertex~$v$ are the immediate
predecessor and successor of this edge with respect to (any) of these orders.
\emph{A first-neighbor path} in~$\Sk$ is a sequence of
directed
edges of~$\Sk$, in which
each edge is followed by one of
its immediate neighbors.

Below we consider graphs with
two types of edges: directed and undirected. We call
such graphs \emph{partially directed}.
The directed and undirected parts (unions of corresponding edges and neighboring vertices) of a partially directed
graph~$\Sk$ are denoted by~$\Skdir$ and~$\Skud$, respectively.

\begin{definition}\label{def.a.skeleton}
	Let~$D$ be a disk.
	An \emph{abstract skeleton} is a partially directed
	embedded graph $\Sk \subset D$,
	disjoint
	from $\partial D$ except for some
	vertices,
	and satisfying the following conditions:
	\newcounter{N5}
	\begin{list}{(\arabic{N5})}{\usecounter{N5}}
		\item\label{Sk.1}
		each vertex is \emph{white}, \emph{black}, or \emph{cross}; the valence of any imaginary white vertex is two, any imaginary black vertex is isolated, any cross vertex is imaginary monovalent and is connected by an incoming edge outgoing from an imaginary white vertex;
		any edge adjacent to a real black vertex (called the \emph{source})
		is outgoing;
		both edges adjacent to an imaginary white vertex are outgoing,
		any isolated black/white vertex belongs to $\Skdir/\Skud$;
		\item\label{Sk.2}
		any immediate neighbor of an incoming
		edge is an outgoing one;
		\item\label{Sk.3}
		$\Sk$ has no first-neighbor cycles;
		\item\label{Sk.4}
		the set of real vertices of~$\Sk$ is non-empty;
		\item\label{Sk.5}
		$b_1+3b=v+z+i$ for each
		region~$R$ of~$\Sk$,
		where $b_1$ is the number of black vertices with a single directed adjacent edge at  $\partial R$, $b$ is the number of isolated (real or imaginary) black vertices in $R$, $v$ is the number of black vertices with two adjacent outgoing edges at $\partial R$, $ z $ is the number of connected components of~$\Skud\cap\partial R$, $i$ is the number of imaginary white vertices on $\partial R$, which are counted twice if they lie on an inner edge of~$R$.
	\end{list}
	
	If additionally
	\newcounter{N6}
	\begin{list}{(\arabic{N6})}{\usecounter{N6}}
		\addtocounter{N6}{\value{N5}}
		\item\label{Sk.8}
		$ \Skdir\cap\Skud=\varnothing $;
		\item\label{Sk.7} at each real vertex there are no		
		directed outgoing edges that are immediate neighbors;
		\item\label{Sk.9} each real white vertex $\Skdir$ has odd valence and is a \emph{sink}, which means that the number of neighboring incoming edges is one more than the number of outgoing edges;
		\item\label{Sk.10} each black vertex is real and
		monovalent (i.e. is a source);
		\item\label{Sk.11}
		the vertices of~$\Skdir$ and $\Skud$ alternate along~$\partial D$,
	\end{list}
	\par\removelastskip
	then $ \Sk $ is called a skeleton \emph{of type~$I$}. 
\end{definition}
  \subsection{Equivalence of abstract skeletons}\label{equivalence}
 Two abstract skeletons
 are called \emph{equivalent} if, up to a homeomorphism $f:D\rightarrow D$
 they can be related by a finite sequence of isotopies
 and the following \emph{elementary moves},
 cf.~\ref{graphmodif}:
 	\begin{itemize}
 		\item[--] \emph{elementary modification}, see \ref{fig.Sk} (a);
 		\item[--] \emph{creating} (\emph{destroying}) a \emph{bridge},
 		see \ref{fig.Sk} (b);
 		the vertex shown in the figure is white,
 		other edges of the skeleton $\Sk$ may also adjoin the vertex;
 		\item[--] \emph{creating} (\emph{removing})  \emph{an undirected edge},
 		see \ref{fig.Sk} (c); the vertex shown in the figure is black real, edges adjacent to it are immediate neighbors, other adjacent directed outgoing edges  may	be present; the edge on the right side of the figure is undirected;
 		\item[--] \emph{\black-in} and its inverse \emph{\black-out}, see
 		\ref{fig.Sk} (d), (e); all vertices shown in the figures
 		are black, other  directed outgoing edges  adjacent to the real vertices	 may	be present in the figure (d);
 		\item[--] \emph{removing/creating a pair of neighboring jump and zigzag},
 		see \ref{fig.Sk} (f);
 		\item[--] \emph{removing/creating a pair of adjacent zigzags},
 		see \ref{fig.Sk} (g);
 		\item[--] \emph{transforming a pair of directed edges into an undirected edge and the inverse transformation},
 		see \ref{fig.Sk} (h).
 		
 	\end{itemize}
 	\begin{figure}[tb]
 		\begin{center}
 			\includegraphics{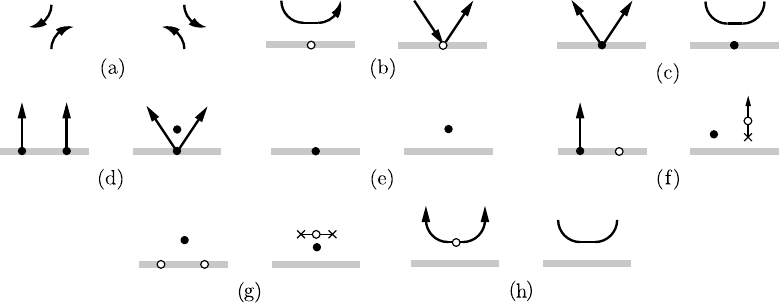}\\
 		\end{center}
 		\caption{Elementary moves of skeletons}\label{fig.Sk}
 	\end{figure}
 	(A move is valid only if the result is again
 	an abstract skeleton.)
 	
 	An equivalence of two abstract skeletons in the same disk
 	with the same set of vertices is called \emph{restricted}
 	if the homeomorphism $f={\rm id}$ and
 	the above isotopies can be chosen to be identical on the vertices.
\subsection{Dotted skeletons}\label{s.skeletons}  
 Intuitively, the dotted skeleton is obtained from a graph $\Gamma$ of some (see the definition of \ref{def.skeleton}) generic trigonal curve, by removing all edges except the dotted ones, gluing the latter at the real \white--vertices and adding imaginary \black--, \white-- and \cross--vertices.
 In this case, the unoriented edges of the skeleton correspond to the junctions of~$\Gamma$.
 
 \begin{definition}\label{def.skeleton}
 	Let $\Gamma \subset D$ be an unramified graph, or a graph obtained from an unramified one by deleting some pairs of adjacent jumps and zigzags, or pairs of adjacent zigzags (see \ref{JZ}, \ref{ZZ}).
 	Let $\bar{D}$ be the disk obtained from~$ D$ by contracting
 	each pillar to a point.
 	The (\emph{dotted}) \emph{skeleton} of~$\Gamma$
 	is a partially directed graph
 	$\Sk=\Sk_\Gamma\subset \bar{D}$
 	obtained from~$\Gamma$ as follows:
 	\begin{itemize}
 		\item[--] each pillar
 		is contracted to a point, which is declared a vertex of the skeleton~$\Sk$, white for the maximal \rm{dotted} segment and black for the maximal \rm{bold} segment;
 		\item[--] each inner \black--/\cross--vertex of~$\Gamma$ is replaced by an inner black/cross--vertex of the skeleton ~$\Sk$;
 		\item[--] each inner \white--vertex of~$\Gamma$ that does not lie on a junction is replaced by an inner white vertex of~$\Sk$, and the adjacent dotted edges 
 		are replaced by oriented skeleton edges directed from this white vertex to the cross or white real vertex;
 		\item [--]
 		each junction is replaced by an unoriented  edge of~$\Sk$, and each of the inner dotted edges of~$\Gamma$ not mentioned above is replaced by an oriented  edge of~$\Sk$ with the orientation obtained from the one of the dotted edge;
 		\item[--] $\Skdir$ ($\Skud$) is the union of black (respectively, white) isolated vertices and the closures of directed (respectively, undirected) edges of $\Sk$ (so that $\Skdir$ and $\Skud$ may intersect at real vertices).
 	\end{itemize}
 	
 \end{definition}
 The following assertions are proved in the same way as the ones of \cite[Propositions 5-7, Theorem 2]{Z21}.
 	The skeleton~$\Sk$ of~$\Gamma$
 	from Definition~\ref{def.skeleton}
 	is an abstract skeleton in the sense of
 	Definition~\ref{def.a.skeleton}.
 
 
 	Any abstract skeleton $\Sk
 	$ is a skeleton
 	of some graph~$\Gamma$
 	in the sense of Definition~\ref{def.skeleton};
 	any two such graphs can be connected by a sequence of
 	isotopies and elementary moves,
 	see~\ref{graphmodif},
 	preserving the skeleton.
 
 	Let $\Gamma_1, \Gamma_2 \subset D$ be the graphs defined by Definition~\ref{def.skeleton};
 	suppose that~$\Gamma_1$ and~$\Gamma_2$
 	have the same pillars.
 	Then $\Gamma_1$ and $\Gamma_2$
 	are related by a restricted equivalence
 	if and only if so are the corresponding
 	skeletons~$\Sk_1$ and~$\Sk_2$.
 
 \begin{theorem} \label{cor.Sk}
 	There is a canonical
 	bijection
 	between the set of rigid isotopy classes 	
 	of almost generic
 	real trigonal curves and the set of equivalence classes
 	of abstract skeletons.
 \end{theorem}
\subsection{Weak equivalence of blocks}  
   By \cite[Proposition 10, Lemma 1]{Z21} for any~$d\ge1$ there exists a unique, up to
   weak equivalence, block $\Gamma\subset D$ of type I and of degree~$3d$.
   \begin{theorem}\label{junction}
   	A block of type II with at least two ovals is weakly equivalent to a graph with a junction.
   \end{theorem}
   \begin{proof}
   	
   	Let $B$ be a block of type II and $d\geq2$ be its degree. 
   	If $B$ has two jumps not separated by any pillar, then a junction is obtained after transformations \ref{fig.Sk} (d), (c) of the skeleton of $B$. Otherwise, each jump has its own adjacent zigzag, and all jumps can be removed using the transformation \ref{fig.Sk} (f), obtaining, according to \cite[Proposition 8]{Z21}, a graph with $d$ (imaginary) \black--vertices, $c$ ovals, and $z$ zigzags, $c\leq d$, $z\geq d$. If in the graph the zigzags alternate with ovals, then the inverse transformation returns to a block of type $I$ according to the same proposition. 
   	Therefore $B$ contains two adjacent ovals, at least on one of the real segments between them there are lying $r\geq2$ zigzags. We apply the transformation \ref{fig.Sk} (g) to $[r/2]$ triples consisting of a pair of these zigzags and a \black--vertex. Depending on whether $r$ is even or not, we transform the skeleton of the resulting graph according to \ref{0} or \ref{1}. 
   	\begin{figure}[tb] 
   		\begin{center} 
   			\scalebox{1.8}{\includegraphics{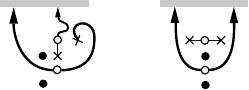}}\\ 
   		\end{center} 
   		\caption{A junction with the type II cubic block}\label{0} 
   	\end{figure} 
   	\begin{figure}[tb] 
   		\begin{center} 
   			\scalebox{1.8}{\includegraphics{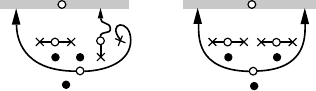}}\\ 
   		\end{center} 
   		\caption{A junction with a block of degree 6}\label{1} 
   	\end{figure} 
   \end{proof}
 \section{Curves on a hyperboloid and trigonal curves}\label{deg9}
 It is well known (see, for example, \cite[3.1.1]{Degt}) that the positive and negative Nagata transformations establish a relationship between (trigonal) curves of bidegree $(m,3)$ on a hyperboloid and proper trigonal curves: a curve $C\subset\Sigma_0=P^1\times P^1$ corresponds to its image $N(C)\subset\Sigma_m$ under the positive Nagata transformation, which inflates the intersection points of~$C $ with  $E_0\subset\Sigma_0$. It is clear that in this case rigidly isotopic curves correspond to rigidly isotopic ones.
 \subsection{Curves $\omega^{\pm}_{inn}$}\label{omega}
 	A curve $C\in\omega^{\pm}_{inn}$ is hyperbolic. Its real scheme can be obtained by combining the real schemes $\langle(3,2)\rangle $ and $\langle(1,1)\rangle$ of nonsingular branches of the curve that intersect transversally at a single point $p$ (see \ref{w.plus}). Therefore, the three intersection points of the curve $C $ with the curve $E_0\subset\Sigma_0=P^1\times P^1$ are real and distinct. Consequently, their images on $N(C)\subset\Sigma_3$ are also real and distinct. To return to~$C$ using a negative Nagata transformation, we need to inflate the singular points of~$N(C)$ and the point $p$ lying on the image of the branch $(3,2)$ (on that arc of it, bounded by the singular points, which has an even intersection multiplicity with the line $y=0$).
 	
 	\begin{figure}
 		\begin{center}
 			\scalebox{0.5}{\includegraphics{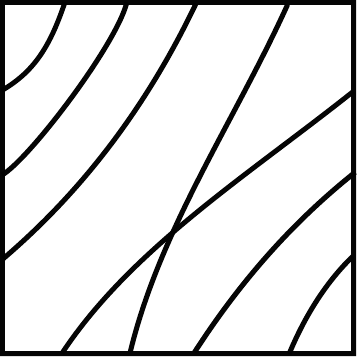}}\\
 		\end{center}
 		\caption{$\omega^{+}_{inn}$}\label{w.plus}
 	\end{figure}
 	\begin{lem}\label{solidcut}
 		The graph of $N(C)$ is equivalent to a graph with a genuine solid cut connecting the imaginary solid vertex with the singular vertices.
 	\end{lem}
 	\begin{proof}
 		Since the degree 9 of the graph of~$N(C)$ is odd, it has an odd number of \white--vertices. On the real segment of the graph bounded by the singular vertices $s_1, s_2$ and having an odd number of \white--vertices, we remove all these \white--vertices except one using \white-in and destroy any bridges that may have appeared. The remaining \white--vertex $w_1$ is connected by real edges $\delta_1,\delta_2$ with $s_1, s_2$ and by a bold edge $\beta_1$ with some \black--vertex $b$. Let $\beta_2, \beta_3$ be other bold edges adjacent to $b$ (see Fig. \ref{s.cut}), and let a \white--vertex $w_2$ be the end of edge $\beta_2$. Making monochrome modifications if necessary, we obtain that $w_2$ is the end of the edge $\beta_3$ and the end of the path $b,\sigma_1,c,\delta_3,w_2$, where $\sigma_1$ is a solid edge lying between $\beta_1$ and $\beta_2$, $c$ is a \cross--vertex, $\delta_3$ is a dotted edge. We connect vertices $b$ and $s_1$ by a solid edge $\sigma_2$, making monochrome modifications if necessary. The region of the obtained graph containing on the boundary the path $s_1,\sigma_2,b,\beta_1,\delta_2,s_2,\sigma_3$, where $\sigma_3$ is a solid edge, is not triangular, therefore during the (solid) monochrome modification of the edges $\sigma_2,\sigma_3$ one can stop in an intermediate position, when an imaginary monochrome vertex and, thus, the desired cut arise. 
 	\end{proof}
 	\begin{figure}
 		\begin{center}
 			\scalebox{1.2}{\includegraphics{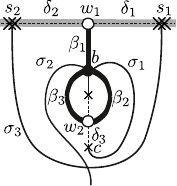}}\\
 		\end{center}
 		\caption{Solid cut}\label{s.cut}
 	\end{figure}
 	\begin{remark}\label{ssp}
 		The image of a singular point of the curve $C$ corresponds to a point of the graph of~$N(C)$ lying on a real segment of the graph bounded by singular vertices $s_1, s_2$ and having an even number of \white--vertices, otherwise there exists a graph transformation that takes $s_1, s_2$ to an imaginary singular vertex, which is impossible, according to what was said at the beginning of Section \ref{omega}.
 	\end{remark}
 	\begin{theorem}
 		Each of the classes $\omega^{+}_{inn}$ and $\omega^{-}_{inn}$ has a  unique wall.
 	\end{theorem}
 	\begin{proof}
 		After the cut specified in Lemma \ref{solidcut}, we obtain a cubic graph of type II and a graph of degree 6 with one oval, which by Theorem \ref{max.inflected} are weakly equivalent to the corresponding blocks. The latter are unique up to weak equivalence by  \cite[Proposition 8, Lemma 1]{Z21}.
 		
 		For a fixed singular fiber, proper trigonal curves symmetric with respect to the $x$-axis have the same graph, but different real schemes $\omega^{+}_{inn}$ and $\omega^{-}_{inn}$.
 	\end{proof}
 	\begin{remark}
 		Curves from the classes $\omega^{\pm}_{out} $, $\tilde{\omega}$ differ from the curve $C\in\omega^{\pm}_{inn}$ only by the choice of the point $p\in N(C)$ (on the arc of the image of the branch $(3,2)$, bounded by singular points, which has an odd intersection multiplicity with the line $y=0$ -- for the classes $\omega^{\pm}_{out} $; on the image of the branch $(1,1)$ -- for the class $\tilde{\omega}$). Therefore, the same arguments prove the connectedness of these classes as well.
 	\end{remark}
 	\subsection{Curves $\alpha^{\pm}_{lp}<l>$} 
 A curve $C\in\alpha^{\pm}_{lp}<l>$ is non-hyperbolic. Its real part contains a (possibly singular) branch realizing the class $(1,\pm2)$ in $H_1(\R X)$, and $l$ ovals if the singular point does not lie on an oval (see Fig. \ref{alpha}). Therefore, either the three intersection points of $C$ with $E_0\subset\Sigma_0=P^1\times P^1$ are real and distinct, or the multiplicity of the intersection of $C$ with $E_0$ at the singular point is three and there is one more point of transversal intersection of these curves. Consequently, the proper trigonal curve $N(C)\subset\Sigma_3$ has the singular fibers listed in \ref{list}. To return to~$C$ using the negative Nagata transformation, we need to inflate the singular points of~$N(C)$ and the point $p$ lying on the image of the branch $(1,\pm2)$ (on that arc of it bounded by the singular points that has an even intersection with the line $y=0$).
 \begin{figure}
 	\begin{center}
 		\scalebox{0.5}{\includegraphics{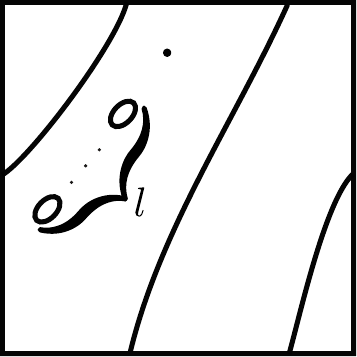}}\\
 	\end{center}
 	\caption{ $\alpha^{+}_{lp}<l>$}\label{alpha}
 \end{figure}
 \begin{remark}\label{gamma_out}
 	For $l=1$, the curve $C$ belongs to type II. A type I curve with a similar real scheme {\rm (see \ref{gamma})} belongs to one of the walls of $\gamma^{\pm}_{out}$. The uniqueness of the latter, as well as of the walls of $\tilde{\gamma}^{\pm}$, $\gamma^{\pm}_{inn} $, follows from the uniqueness of the block of degree 9, having type~I, {\rm (see \cite[Proposition 10, Lemma 1]{Z21})} since these walls differ only in the choice of the point $p\in N(C)$ (on the arc of the image of the branch $(1,\pm2)$, bounded by singular points, which has an even/odd intersection multiplicity with the line $y=0$ -- for the walls $\gamma^{\pm}_{out}$/$\tilde{\gamma}^{\pm}$; on the oval -- for the walls $\gamma^{\pm}_{inn} $). \end{remark}
 \begin{figure}
 	\begin{center}
 		\scalebox{0.5}{\includegraphics{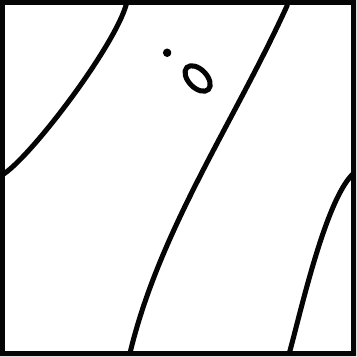}}\\
 	\end{center}
 	\caption{ $\gamma^{+}_{out}$}\label{gamma}
 \end{figure}
 \begin{lem}\label{isolp-ts}
 	There exists a nodal-cuspidal curve with two non-degenerate isolated singular points that is rigidly isotopic to $N(C)$.
 \end{lem}
 \begin{proof}
 	The singular points of $N(C)$ correspond to the singular vertices of its graph. Let one of the singular vertices $s$ be non-isolated. If it is connected by a real dotted edge to a (non-singular) \cross--vertex, then $s$ can be turned into an isolated vertex using the transformation \ref{isol}.
 	
 	Now let $s$ be a singular vertex connected to two \cross--vertices by real dotted segments containing an odd number of \white--vertices. Using \white-in, we leave one \white--vertex on each of these segments. 
 	By Theorem \ref{s.max.inflected}, the graph of $N(C)$ is unramified, so the result $\Gamma$ of its perturbation, which turns singular vertices into ovals, 
 	splits into blocks after genuine cuts along the dotted edges (see \ref{CR}). Therefore, the solid segment obtained from $s$ lies in one of the blocks between the two zigzags, so the block is of type II and by Theorem \ref{junction}  has at most one oval. Therefore, there is a jump near one of the zigzags. Returning to the graph of~$N(C)$ and applying the transformations \ref{A1} and \ref{isol} to the fragment of the graph containing the \black--vertex of this jump, we obtain an isolated singular vertex.
 	
 	Finally, the last possible case that prevents us from applying the transformation \ref{isol} at once is when both singular vertices $s_1,s_2$ lie between two \cross--vertices on a dotted segment and divide it into three segments containing an odd number of \white--vertices (since Remark \ref{ssp} is obviously true also for~$C$,  there is an odd number of \white--vertices between $s_1,s_2$). Using \white-in, we leave one \white--vertex on each of these segments. The same arguments as in the previous case show that there is a jump near one of the  \cross--vertices, so applying the transformation \ref{A1} removes the \white--vertex between $s_1,s_2$, resulting in a graph of a curve that does not lie in $\alpha^{\pm}_{lp}<l>$ by virtue of  Remark \ref{ssp}.
 \end{proof}
 Using the proven lemma, we will further assume that the singular vertices of the graph of~$N(C)$ are isolated. Denote by $\Gamma$ the result of the perturbation of this graph, turning the singular vertices into ovals.
 \begin{lem}\label{1white}
 	
 	
 	The graph of the curve $N(C)$ is weakly equivalent to a graph with a single real \white--vertex.
 	\end{lem}
 	\begin{proof}
 	By Lemma \ref{isolp-ts} the singular vertices of the graph can be considered isolated.
 	%
 	According to Remark \ref{gamma_out}, the graph $\Gamma$ cannot be a block of type I, therefore, by Theorem \ref{junction}, the graph is a junction of three cubic blocks or a cubic block and a block of degree 6 (see \ref{blocks6}). The transformations \ref{JZ}, \ref{ZZ} in the left and right blocks in the first case and in both blocks in the second, followed by contraction of two ovals to the singular points, yield the desired graph.
 	\footnote{The same considerations give the rigid isotopy classification of non-singular real trigonal curves of genus 4 on a quadratic cone. After blowing up the vertex of the cone, such a curve yields a proper curve on $\Sigma_2$ whose graph is either a block of degree 6 or a junction of two cubic blocks. Curves obtained from each other by reflection about the $OX$ axis have the same graph, but symmetric real schemes. A rigid isotopy class is determined by the complex scheme; their number is 11: for $l=0$ the scheme of a hyperbolic curve (type I) and a scheme of type II, for $l=2$ the scheme of type I and two schemes of type II symmetric to each other, for $l=1, 3, 4$  two schemes symmetric to each other.}

 	\end{proof}
 	\begin{figure}
 	\begin{center}
 		\includegraphics{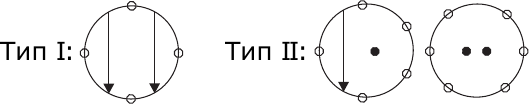}\\
 	\end{center}
 	\caption{Blocks of degree 6 up to weak equivalence}\label{blocks6}
 	\end{figure}
 	\begin{theorem}
 	The wall $\alpha^{\pm}_{lp}<l>$ is uniquely determined by the complex scheme.
 	\end{theorem}
 	\begin{proof}
 	After the transformation specified in Lemma \ref{1white}, singular vertices can be interchanged with any ovals according to Proposition \ref{s-oval}. Therefore, it suffices to prove the theorem for the graph $\Gamma$.
 	Consider all values of~$l$.
 	\begin{enumerate}
 		\item For $l=5$, by \cite[Lemma 1]{Z21}, 
 		$\Gamma$ is unique up to weak equivalence as the graph of an $M$-curve.
 		\item For $l=4$, by \cite[6.4.2]{DIK08} and \cite[Lemma 1]{Z21}, 
 		 $\Gamma$ is unique up to weak equivalence as the graph of an $(M-1)$-curve.
 		\item For $l=3$, as was already said in the proof of Lemma \ref{1white}, 
 		$\Gamma$ is a junction of a cubic block $K$ and a graph $B$ of degree 6. Moreover, if $\Gamma$ belongs to type~I, $K$ and $B$ are blocks of type I, unique up to weak equivalence by \cite[Proposition 8, Lemma 1]{Z21}. If $\Gamma$ is of type II, $B$ contains a junction by Theorem \ref{junction}, otherwise it would be a block of type I with two ovals, and so $K$ would also be a block of type I. Therefore $\Gamma$ is a junction of three cubic blocks, one of type I and two of type II. The union of a block of type I and a block of type II yields an $(M-1)$-curve graph, in which the blocks can be permuted according to \cite[6.4.2]{DIK08}. Therefore $\Gamma$ is equivalent to a graph with a central cubic block of type I, and so is unique. 
 		\item For $l=2$, as well as for $l=3$, 
 		$\Gamma$ is a junction of a cubic block $K$ and a graph $B$ of degree 6. If $K$ is of type I, the following transformations of the skeleton of $\Gamma$ allow to obtain a junction with the ends in neighboring ovals, turning $K$ into a cubic block of type II: transformations (f), (g) \ref{fig.Sk} applied to $K$ and $B$, transformation (h) \ref{fig.Sk} applied to contact, transformation (a) \ref{fig.Sk} applied to edges $e_1,e_2$ on \ref{l=2}, transformation \ref{0} applied to edges $e'_1,e_3$ on \ref{l=2}. Then by Theorem \ref{junction} the new graph $B$ is a junction of cubic blocks of type II and therefore $\Gamma$ is unique.
 		\item For $l=1$, $\Gamma$ is a junction of a cubic block $K$ and a block $B$ of degree 6. If $K$ is of type I, the same transformations of the skeleton of $\Gamma$ as for $l=2$ carry the oval from $K$ to $B$, placing it near either end of the junction. The inverse transformations and the transformation \ref{fig.Sk} (g) yield the skeleton shown in \ref{l=1}. Therefore, the uniqueness of $\Gamma$ follows from its equivalence to the junction of a cubic block of type I and a block of degree 6 without ovals.
 		\item For $l=0$, $\Gamma$ is a junction of a cubic block $K$ of type II and a block $B$ of degree~6 without ovals, and is therefore unique.
 		
 	\end{enumerate}
 	For a fixed singular fiber, proper trigonal curves symmetric with respect to the $x$ axis have the same graph, but different real schemes $\alpha^{+}_{lp}<l>$, $\alpha^{-}_{lp}<l>$.
 	\end{proof}
 	
 	\begin{figure}
 	\begin{center}
 		\scalebox{1.5}{\includegraphics{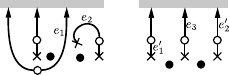}}\\
 	\end{center}
 	\caption{
 		A junction with the cubic block of type I}\label{l=2}
 		\end{figure}
 		\begin{figure}
 	\begin{center}
 		\scalebox{1.5}{\includegraphics{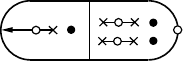}}\\
 	\end{center}
 	\caption{The skeleton of the graph $\Gamma$ with $l=1$}\label{l=1}
 	\end{figure}
 	\begin{remark}
 	Curves from the classes $\tilde{\alpha}<l>$, $\alpha^{\pm}_{ov}<l>$, differ from a curve $C\in\alpha^{\pm}_{lp}<l>$ only by the choice of the point $p\in N(C)$ (for the classes $\tilde{\alpha}<l>$, the point $p$ is on the arc of the image of branch $(1,\pm2)$, bounded by singular points, that has odd intersection multiplicity with the line $y=0$; for the classes $\alpha^{\pm}_{ov}<l>$, the point $p$ is on an oval). Therefore, the same arguments prove connectedness of these  classes as well.
 	\end{remark}

\end{document}